\documentclass{article}

\usepackage{amssymb}
\usepackage{amsfonts}
\usepackage{amsmath}
\usepackage{amsthm}
\usepackage{url}
\usepackage{color}
\usepackage[small,nohug,heads=LaTeX]{diagrams}
\diagramstyle[labelstyle=\scriptstyle]
\usepackage{graphicx}

\newcommand{\bq}{\begin{quote}}
\newcommand{\eq}{\end{quote}}
\newcommand{\bi}{\begin{itemize}}
\newcommand{\ei}{\end{itemize}}
\newcommand{\bd}{\begin{description}}
\newcommand{\ed}{\end{description}}
\newcommand{\ben}{\begin{enumerate}}
\newcommand{\een}{\end{enumerate}}
\newcommand{\bbm}{\begin{bmatrix}}
\newcommand{\ebm}{\end{bmatrix}}
\newcommand{\bea}{\begin{eqnarray*}}
\newcommand{\eea}{\end{eqnarray*}}

\newtheorem{theorem}{Theorem}[section]

\newtheorem{lemma}[theorem]{Lemma}

\newtheorem{remark}[theorem]{Remark}
\newtheorem{corollary}[theorem]{Corollary}
\newtheorem{algorithm}[theorem]{Algorithm}

\def\Alt{{\rm Alt}}
\def\Sym{{\rm Sym}}
\def\2G2{\ensuremath{^2{\rm G}_2}}
\def\sl{{\rm{SL}}}
\def\gl{{\rm{GL}}}

\def\psl{{\rm{PSL}}}
\def\pgl{{\rm{PGL}}}

\def\ppsl{ ( {\rm{P}} ) {\rm{SL}}}

\def\qpone{q \equiv 1  \bmod 4}
\def\qmone{q \equiv -1 \bmod 4}

\definecolor{darkgreen}{rgb}{0,0.6,0}

\title{Construction of some subgroups in black box groups $\pgl_2(q)$ and $\ppsl_2(q)$}

\author{Alexandre Borovik\footnote{School of Mathematics, University of Manchester, UK; alexandre@borovik.net} \   and \c{S}\"{u}kr\"{u} Yal\c{c}\i nkaya\footnote{Nesin Mathematics Village, Izmir, Turkey; sukru.yalcinkaya@gmail.com}}

\date{}

\begin{document}

\maketitle

\begin{abstract}
For the black box groups $X$ encrypting $\pgl_2(q)$, $q$ odd, we propose an algorithm constructing a subgroup encrypting $\Sym_4$ and subfield subgroups of $X$. We also present the analogous algorithms for black box groups encrypting $\ppsl_2(q)$.
\end{abstract}

\section{Introduction}

It becomes apparent that the groups $\psl_2(q)$ and $\pgl_2(q)$, $q$ odd, play a fundamental role in the constructive recognition of black box groups of Lie type of odd characteristic \cite{suko14A}. This paper provides the fundamentals for the algorithms presented in $\cite{suko14A}$, that is, we present polynomial time Las Vegas algorithms constructing black box subgroups encrypting $\Sym_4$ and subfield subgroups of black box groups encrypting $\pgl_2(q)$. We also describe the corresponding algorithms for the black box groups encrypting $\ppsl_2(q)$.

In this paper, we use description of  $\pgl_2(q)$ as the semidirect product $\pgl_2(q)=\psl_2(q)\rtimes  \langle \delta \rangle$ where $\delta$ is a diagonal automorphism of $\psl_2(q)$ of order $2$. We refer the reader to \cite[Chapter XII]{dickson1958} or \cite[Chapter 3.6]{suzuki1982} for the subgroup structure of $\ppsl_2(q)$.

A black box group $X$ is a black box (or an oracle, or a device, or an algorithm) operating with $0$-$1$ strings of bounded length which encrypt (not necessarily in a unique way) elements of some finite group $G$. The functionality of the black box is specified by the following axioms: the black box

\begin{itemize}
\item[\textbf{BB1}] produces strings encrypting random elements from $G$;
\item[\textbf{BB2}]computes a string encrypting the product of two group elements given by strings or a string encrypting the inverse of an element given by a string; and
\item[\textbf{BB3}] compares whether two strings encrypt the same element in $G$.\end{itemize}

In this setting we say that black box group $X$ \emph{encrypts} $G$.

A typical example is provided by a group $G$ generated in a big matrix group $\gl_n(p^k)$ by several matrices $g_1,\dots, g_l$. The product replacement algorithm \cite{celler95.4931} produces  a sample of (almost) independent elements from a distribution on $G$ which is close to the uniform distribution (see the discussion and further development in \cite{babai00.627, babai04.215, bratus99.91, gamburd06.411, lubotzky01.347, pak00.476, pak01.301, pak.01.476, pak02.1891}). We can, of course, multiply, invert, compare matrices. Therefore the computer routines for these operations together with the sampling of the product replacement algorithm run on the  tuple of generators $(g_1,\dots, g_l)$ can be viewed as a black box $X$ encrypting the group $G$. The group $G$ could be unknown---in which case we are interested in its isomorphism type---or it could be known, as it happens in a variety of other black box problems.

Unfortunately, an elementary task of determining the order of a string representing a group element involves either integer factorisation or discrete logarithm. Nevertheless black box problems for matrix groups have a feature which makes them more accessible:

\begin{itemize}
\item[\textbf{BB4}] We are given a \emph{global exponent} of $X$, that is, a natural number $E$ such that it is expected that $x^E = 1$ for all elements $x \in X$ while computation of $x^E$ is computationally feasible.
\end{itemize}

Usually, for a black box group $X$ arising from a subgroup in the ambient group $\gl_n(p^k)$, the exponent of $\gl_n(p^k)$ can be taken for a global exponent of $X$.

\begin{quote}
\emph{In this paper, we assume that all our black box groups satisfy assumptions BB1--BB4.}
\end{quote}

A randomized algorithm is called \textit{Las Vegas} if it always returns a positive answer or fails with some probability of error bounded by the user, see \cite{babai97.1} for a discussion of randomized algorithms.

We refer reader to \cite{suko12F} for a more detailed discussion of black box groups and the nature of the problems in black box group theory.

Our principal result is the following.

\begin{theorem}\label{pgl2}
Let $X$ be a black box group encrypting $\pgl_2(p^k)$ where $p$ is a known odd prime and $k$ is unknown. Then there exists a Las Vegas algorithm constructing a subgroup encrypting $\Sym_4$ and, if $p \neq 5$, a black box  subfield subgroup $\pgl_2(p)$.

The running time of the algorithm is $O(\xi(\log \log q +1) +\mu(k\log \log q \log q + \log q))$, where $\mu$ is an upper bound on the time requirement for each group operation in $X$ and $\xi$ is an upper bound on the time requirement, per element, for the construction of random elements of $X$.
\end{theorem}

\begin{corollary}\label{psl2}
Let $X$ be a black box group encrypting $\ppsl_2(p^k)$ where $p$ is a known odd prime and $k$ is unknown. Then there exists a Las Vegas algorithm constructing a subgroup encrypting
\begin{itemize}
\item[{\rm (i)}]  $\Alt_4$ or $\Sym_4$ when $q\equiv \pm 3 \mbox{ mod } 8$ or  if $q\equiv \pm 1 \mbox{ mod } 8$, respectively, if $X\cong \psl_2(p^k)$, and the normalizer $N$  of a quaternion group, if $X\cong \sl_2(p^k)$; and
\item[{\rm (ii)}] if $p \neq 5,7$ a subfield subgroup $\ppsl_2(p)$.
\end{itemize}
The running time of the algorithm is $O(\xi(\log \log q +1) +\mu(k\log \log q \log q + \log q))$, where $\mu$ is an upper bound on the time requirement for each group operation in $X$ and $\xi$ is an upper bound on the time requirement, per element, for the construction of random elements of $X$.
\end{corollary}

\begin{corollary}\label{corall}
Let $X$ be a black box group encrypting $\pgl_2(p^k)$ or $\ppsl_2(p^k)$ where  $p$ is a known odd prime with known $k$. Then, for any divisor $a>1$ of $k$, there exists a Las Vegas algorithm constructing a black box subgroup encrypting a subfield subgroup $\pgl_2(p^a)$ or $\ppsl_2(p^a)$, respectively.
 \end{corollary}

\section{Subfield subgroups and $Sym_4$ in $\pgl_2(p^k)$}\label{sec:alt}

Let $G\cong \pgl_2(q)$, $q=p^k$, $p$ an odd prime.
Note that $G$ has two conjugacy classes of involutions, say $\pm$-type involutions, where the order of the centralizer of a $+$-type involution is $2(q-1)$ and the order of the centralizer of a $-$-type involution is $2(q+1)$. Notice that $C_G(i)=T\rtimes \langle w \rangle$ where $T$ is a torus of order $(q\pm 1)$ and $w$ is an involution inverting $T$.  Throughout the paper, we consider the involutions of $+$-type if $\qpone$ and $-$-type if $\qmone$ so that the order of the torus $T$ is always divisible by $4$; we call them involutions of \emph{right type}.

We set $5$-tuple
\begin{equation}\label{set:1}
(i,j,x,s,T)
\end{equation}
where $i\in G$ is an involution of right type, $T<G$ is the torus in $C_G(i)$, $j\in G$ is an involution of right type which inverts  $T$, $x\in G$ is an element of order $3$ normalising $\langle i,j\rangle$ and $s \in T$ is an element of order $4$. We also set $k=ij$ and note that $k$ is also of right type. Clearly $V=\langle i,j\rangle$ is a Klein 4-subgroup and $\langle i,j,x\rangle \cong \Alt_4$.  Moreover, we have  $\langle i,j,x,s \rangle \cong \Sym_4$. 

An alternative and slightly easier construction of $\Sym_4$ in $\pgl_2(q)$ is as follows. Let $i,j \in G\cong \pgl_2(q)$ be involutions of right type where $j$ inverts the torus in $C_G(i)$, choosing the elements $t_i,t_j$  of order $4$ in the tori in $C_G(i)$ and $C_G(j)$, respectively, we have $\Sym_4 \cong \langle t_j, t_j\rangle$. However, such a construction of $\Sym_4$ in $\pgl_2(q)$ does not cover the corresponding construction of $\Alt_4$ in $\psl_2(q)$ when $q\equiv  \pm 3 \mbox{ mod } 8$, see Remark \ref{rem:alt4} (1). For the sake of completeness, we follow the setting in (\ref{set:1}).

\begin{remark}\label{rem:alt4}
{\rm
\
\begin{itemize}
\item [(1)]If $G\cong \psl_2(q)$, then $G$ has only one conjugacy classes of involutions and $C_G(i) = T\rtimes \langle w\rangle$ where $|T|=(q -1)/2$ if $\qpone$, and $|T|=(q +1)/2$ if $\qmone$. Therefore $T$ contains element of order 4  if and only if $q\equiv  \pm 1 \mbox{ mod } 8$. Thus, we can construct subgroups isomorphic to  $\Sym_4$ in $G$ precisely when $q\equiv  \pm 1 \mbox{ mod } 8$. Otherwise, the subgroup $\Alt_4$ will be constructed. We shall note here that  $\Alt_4$ or $\Sym_4$ are maximal subgroups of $\psl_2(p)$ if $p\equiv \pm 1 \mbox{ mod } 8$ or $p\equiv \pm 3 \mbox{ mod } 8$, respectively \cite[Proposition 4.6.7]{kleidman1990}.

\item[(2)] If $G\cong \sl_2(q)$, then $i, j$ are pseudo-involutions (whose squares are the central involution in $\sl_2(q)$) and $V=\langle i,j \rangle$ is a quaternion group. Moreover, if $q\equiv  \pm 3 \mbox{ mod } 8$ ($q\equiv  \pm 1 \mbox{ mod } 8$, respectively), the subgroup $\langle i,j,x \rangle$ ($\langle i,j,s,x\rangle$, respectively) is $N_G(V)$, where $s$ is an element of order $8$ in $C_G(i)$.
\end{itemize}
}
\end{remark}

The main ingredient of the algorithm in the construction of $\Sym_4$ and subfield subgroups of $G\cong \pgl_2(q)$ is to construct an element $x\in G$ of order 3 permuting some mutually commuting involutions $i,j,k \in G$ of right type. The following lemma provides explicit construction of such an element.

\begin{lemma}\label{elt:3}
Let $G\cong \pgl_2(q)$, $q$ odd, $i,j,k$ mutually commuting involutions of right type. Let $g \in G$ be an arbitrary element. Assume that $h_1=ij^g$ has odd order $m_1$ and set $n_1=h_1^{\frac{m_1+1}{2}}$ and $s=k^{gn_1^{-1}}$. Assume also that $h_2=js$ has odd order $m_2$ and set  $n_2=h_2^{\frac{m_2+1}{2}}$. Then  the element $x=gn_1^{-1}n_2^{-1}$ permutes $i,j,k$ and $x$ has order $3$.
\end{lemma}

\begin{proof}
Observe first that $i^{n_1}=j^g$ and $j^{n_2}=s$. Then, since $s=k^{gn_1^{-1}}$, we have $j^{n_2}=k^{gn_1^{-1}}$. Hence $j=k^{gn_1^{-1}n_2^{-1}}=k^x$. Now, we prove that $j^x=i$. Since $j^{gn_1^{-1}}=i$, we have $j^{x}=j^{gn_1^{-1}n_2^{-1}}=i^{n_2^{-1}}$. We claim that $h_2 \in C_G(i)$, which implies that $n_2 \in C_G(i)$, so $j^x=i^{n_2^{-1}}=i$. Now, since $j \in C_G(i)$, $h_2=js\in C_G(i)$ if and only if $s =k^{gn_1^{-1}}\in C_G(i)$. Recall that $i^{n_1}=j^g$. Therefore $s \in C_G(i)$ if and only if $k^g \in C_G(j^g)$, equivalently, $k \in C_G(j)$ and the claim follows. It is now clear that $i^x=k$ since $ij=k$. It is clear that $x \in N_G(V)$ where $V=\langle i,j \rangle$ and $x$ has order 3.
\end{proof}

\begin{lemma}\label{elt:3:prob}
Let $G$, $h_1$ and $h_2$ be as in Lemma {\rm \ref{elt:3}}. Then the probability that $h_1$ and $h_2$ have odd orders is bounded from below by $1/2-1/2q$.
\end{lemma}

\begin{proof}
We first note that the subgroup $\langle i,x\rangle \cong \Alt_4$ is a subgroup of $L\leq G$ where $L\cong \psl_2(p)$, so the involutions $i,j,k$ belong to a subgroup isomorphic to $ \psl_2(q)$. Therefore it is enough to compute the estimate in $H \cong \psl_2(q)$. Notice that all involutions in $H$ are conjugate. Therefore the probability that $h_1$ and $h_2$ have odd orders is the same as the probability of the product of two random involutions from $H$ to be of odd order.

We denote by $a$ one of these numbers $(q\pm 1)/2$ which is odd and by $b$ the other one. Then $|H|=q(q^2-1)/2=2abq$ and $|C_H(i)|=2b$ for any involution $i\in H$. Hence the total number of involutions is
\[
\frac{|H|}{|C_H(i)|}=\frac{2abq}{2b}=aq.
\]

Now we shall compute the number of pairs of involutions $(i,j)$ such that their product $ij$ belongs to a torus of order $a$. Let $T$ be a torus of order $a$. Then $N_H(T)$ is a dihedral group of order $2a$. Therefore the involutions in $N_H(T)$ form the coset $N_H(T)\backslash T$ since $a$ is odd. Hence, for every torus of order $a$, we have $a^2$ pairs of involutions whose product belong to $T$. The number of tori of order $a$ is $|H|/|N_H(T)|=2abq/2a=bq$. Hence, there are $bqa^2$ pairs of involutions whose product belong to a torus of order $a$. Thus the desired probability is
\[
\frac{bqa^2}{(aq)^2}=\frac{b}{q}\geqslant \frac{q-1}{2q}=\frac{1}{2}-\frac{1}{2q}.
\]
\end{proof}

For the subfield subgroups isomorphic to $\pgl_2(p^a)$ of $G\cong \pgl_2(q)$, $q=p^k$, $p$ an odd prime, we extend our setting in (\ref{set:1}) and set $6$-tuple
\begin{equation}\label{set:2}
(i,j,x,s,r,T)
\end{equation}
where $r\in T$ has order $(p^a\pm 1)$ where $(p^a\pm 1)/2$ is even. Notice that if  $a$ is a divisor of $k$, then the torus $T$ contains an element $r$ of order $(p^a\pm 1)$ where $(p^a\pm 1)/2$ is even. The following lemma provides explicit generators of the subfield subgroups of $G$.

\begin{lemma}\label{the:killer}
Let $G\cong \pgl_2(q)$, $q=p^k$ for some $k\geqslant 2$ and $(i,j,x,s,r,T)$ be as in {\rm (\ref{set:2})}. Then $\langle r, x \rangle \cong \pgl_2(p^a)$ except when $a=1$ and $p=5$.
\end{lemma}

\begin{proof}
Let $L=\langle i,j,x,s \rangle \cong \Sym_4\cong \pgl_2(3)$. Observe that $L$ is a subgroup of some $H\leqslant G$ where $H\cong \pgl_2(p)$. Now assume first that $a=1$. Since $r \in C_G(i)$,  the order of the subgroup $T\cap H$ is  $p\pm 1$. Since $T$ is cyclic, it has only one subgroup of order $p\pm 1$ so $r\in H$. Thus $\langle r,x\rangle \leqslant H$. By the subgroup structure of $\pgl_2(p)$, the subgroup $L\cong \Sym_4$ is either a maximal subgroup or contained in a maximal subgroup of $H$ isomorphic to $\Sym_4 \rtimes \langle \delta \rangle$ where $\delta$ is a diagonal automorphism of $\psl_2(q)$. Hence, if $|r|\geqslant 7$, or equivalently $p\geqslant 7$, then we have $\langle r,x\rangle =H$ since such a maximal subgroup does not contain elements of order bigger than 7. As we noted above, if $p=3$, then $L\cong \Sym_4\cong \pgl_2(3)$.

Observe that if $a>1$ and $a$ is a divisor of $k$, then an element $r$ of order $p^a\pm 1$, where $(p^a\pm1)/2$ is even, belongs to a subgroup $H\cong \pgl_2(p^a)$  hence the lemma follows from the same arguments above.
\end{proof}

\begin{remark}\label{rem:the:killer}
\
{\rm
\begin{itemize}
\item[(1)] Following the notation of Lemma \ref{the:killer}, observe that if $a=1$ and $p=5$, then $|r|=4$ and  $\langle r,x\rangle \cong \Sym_4$.
\item[(2)] If $G\cong \psl_2(q)$, then, there is one more exception in the statement of Lemma \ref{the:killer}, that is, $a=1$ and $p=7$. This extra exception arises from the fact that the torus $T\cap H$ in the proof of Lemma \ref{the:killer} has order $(p\pm 1)/2$ and the element $r$ has order $4$. Again, we are in the situation that $\langle r,x\rangle\cong \Sym_4 <\psl_2(7)$.
\item[(3)]If $G\cong \sl_2(q)$, then, by considering the pseudo-involutions, the same result in Lemma \ref{the:killer} holds with the exceptions $a=1$ and $p=5$ or $7$.
\end{itemize}
}
\end{remark}

\section{The algorithm}\label{sec:algo}
In this section we present an algorithm for the black box group encrypting $\pgl_2(p^k)$ and the corresponding algorithm for the groups $\ppsl_2(p^k)$ follows from Remarks \ref{rem:alt4} and \ref{rem:the:killer}.

In order to cover the algorithm in Corollary \ref{corall}, we assume below that a divisor $a$ of $k$ is given as an input. Observe that such an input is not needed for the construction of a subfield subgroup $\pgl_2(p)$.

\begin{algorithm}\label{algorithm:psl2} Let $X$ be a black box group isomorphic to $\pgl_2(q)$, $q=p^k$, $p$ an odd prime.
\begin{itemize}
\item[Input:]
\begin{itemize}
\item[$\bullet$] A set of generators of $X$.
\item[$\bullet$] The characteristic $p$ of the underlying field.
\item[$\bullet$] An exponent $E$ for $X$.
\item[$\bullet$] A divisor $a$ of $k$.
\end{itemize}
\item[Output:]
\begin{itemize}
\item[$\bullet$] A black box subgroup encrypting $\Sym_4$.
\item[$\bullet$] A black box subgroup encrypting $\pgl_2(p^a)$ except when $a=1$ and $p=5$.
\end{itemize}
\end{itemize}
\end{algorithm}

Outline of Algorithm \ref{algorithm:psl2} (a more detailed description follows below):
\begin{itemize}
\item[1.] Find the size of the field $q=p^k$ (This step is not needed for Corollary \ref{corall}).
\item[2.] Construct an involution $i\in X$ of right type from a random element together with a generator $t$ of the torus $T< C_X(i)$ and a Klein 4-group $V=\langle i,j\rangle$ in $X$ where $j$ is an involution of right type.
\item[3.] Construct an element $x$ of order 3 in $N_X(V)$.
\item[4.]  Set $s=t^{|T|/4}$ and deduce that $\langle s,x \rangle \cong \Sym_4$.
\item[5.] Set $r=t^{|T|/(p^a\pm 1)}$ where $(p^a \pm 1)/2$ is even and deduce that $\langle r,x\rangle\cong \pgl_2(p^a)$ except when $a=1$ and $p=5$.
\end{itemize}

Now we give a more detailed description of Algorithm \ref{algorithm:psl2}.

\bd
\item[{\bf Step 1:}] We compute the size $q$ of the underlying field by Algorithm 5.5 in \cite{yalcinkaya07.171}.\\

\item[{\bf Step 2:}] Let $E=2^km$ where $(2,m)=1$. Take an arbitrary element $g\in X$. If the order of $g$ is even, then the last non-identity element in the following sequence is an involution
\[
1\neq g^m, \, g^{2m},\, g^{2^2m}, \ldots, g^{2^km}=1.
\]
Let $i\in X$ be an involution constructed as above. Then, we construct $C_X(i)$ by the method described in \cite{borovik02.7, bray00.241} together with the result in \cite{parker10.885}. To check whether $i$ is an involution of right type, we construct a random element $g\in C_X(i)$ and consider $g^{q\pm1}$. If $|g|>2$ and $g^{q+1}\neq 1$, then $i$ is of $+$-type. We follow the analogous process to check whether $i$ is of $-$-type. We have $C_X(i) = T\rtimes \langle w \rangle$ where $T$ is a torus of order $q \pm 1$ and $w$ is an involution which inverts $T$. Observe that the coset $Tw$ consists of involutions inverting $T$, so half of the elements of $C_X(i)$ are the involutions inverting $T$ and half of the involutions in $Tw$ are of the same type as $i$. We check whether $j$ has the same type as $i$ by following the same procedure above. Let $j \in C_X(i)$ be such an involution, then, clearly, $V=\langle i,j \rangle$ is a Klein 4-group. For the construction of a generator of $T$, notice that a random element of $C_X(i)$ is either an involution inverting $T$ or an element of $T$ and, by \cite{mitrinovic1996}, the probability of finding a generator of  a cyclic group of order $q \pm 1$ is at least $O(1/\log \log q)$. Since $|T|$ is divisible by $4$, we can find an element $t \in C_X(i)$ such that $t^2\neq 1$ and $t^{|T|/2} \neq 1$ with probability at least $O(1/\log \log q)$ and such an element is a generator of $T$.

\item[{\bf Step 3:}] By Lemmas \ref{elt:3} and \ref{elt:3:prob}, we can construct an element $x$ of order 3 normalizing $V=\langle i,j\rangle$ with probability at least $1/2-1/2q$.
\item[{\bf Step 4:}]  Since the order $T$ is divisible by 4, we set $s=t^{|T|/4}$ and we can deduce that $\langle s,x \rangle \cong \Sym_4$ from the discussion in the beginning of Section \ref{sec:alt}.
\item[{\bf Step 5:}] It follows from Lemma \ref{the:killer} that the subgroup $\langle r,x \rangle$ encrypts a black box group $\pgl_2(p^a)$ except when $a=1$ and $p=5$.
\ed

Following the arguments in Remarks \ref{rem:alt4} and \ref{rem:the:killer}, we have the corresponding algorithms for the black box groups encrypting $\ppsl_2(q)$.

\subsection{Complexity}\label{complexity}

Let $\mu$ be an upper bound on the time requirement for each group operation in $X$ and $\xi$ an upper bound on the time requirement, per element, for the construction of random elements of $X$.

We outline the running time of Algorithm \ref{algorithm:psl2} for each step as presented in the previous section. For simplicity, we assume that  $E=|X|=|\pgl_2(q)|=q(q^2-1)$.

\bd
\item[{\bf Step 1}]  First, random elements in $X$ belong to a torus of order $q-1$ or $q+1$ with probability at least $1-O(1/q)$. Then, in each type of tori, by \cite{mitrinovic1996}, we can find an elements of order $q-1$ and $q+1$ with probability $c/\log \log q$ for some constant $c$. Therefore, producing $m=O(\log \log q)$ elements $g_1, \ldots, g_m$, we assume that one of $g_i$ has order $q-1$ and $g_j$ has order $q+1$. Now, checking each $g_i^{p(p^{2\ell}-1)} =1$ involves at most $\log p^{2\ell +1}$ group operations making the overall cost to determine the exact power of $p$ involving in $q=p^k$,
\[
\sum_{\ell=1}^k \log(p^{2\ell+1})=\log p^{k^2+2k}=(k+2)\log q.
\]
Hence the size of the field can be computed in time $O(k \mu \log \log q \log q + \xi \log \log q)$.
\item[{\bf Step 2}]  By \cite[Corollary 5.3]{isaacs95.139}, random elements in $X$ have even order with probability at least $1/4$. Then, construction of an involution $i$ from a random element  and checking whether an element of the form $ii^g$ has odd order for a random element involves constant number of construction of a random element in $X$ and $C_X(i)$ and $\log E \leq \log q^3$ group operations by repeated square and multiply method. Checking whether an involution is of desired type involves $\log E$ group operations. By \cite{mitrinovic1996}, we can find a generator for the torus $T \leq C_X(i)$ with probability $O(1/\log \log q)$ and checking whether it is indeed a generator of $T$ involves $\log q$ group operations. Hence we can construct involutions $i,j$ of desired type and a generator $t$ of the torus $T$ in time $O(\xi (1+\log \log q) + \mu \log \log q \log q )$.
\item[{\bf Step 3}] By Lemma \ref{elt:3:prob} the elements $h_1=ij^g$ and $h_2=jk^{gu_1^{-1}}$ have odd orders $m_1$ and $m_2$ with probability $1/2-1/2q$. Checking both elements for having odd order and construction of elements $h_1^{\frac{m_1+1}{2}}$ and $h_2^{\frac{m_2+1}{2}}$ involves $\log E$ group operations making overall cost $O(\xi+\mu\log q)$ to construct an element $x$ of order $3$ permuting the involutions $i,j,k$ of right type.
\item[{\bf Step 4}] The element $s$ can be constructed in time  $O(\mu \log q)$.
\item[{\bf Step 5}] The element $r$ can be constructed in time  $O(\mu \log q)$.
\ed
Combining the running times of the steps above, the overall running time of the algorithm for the construction of $\Sym_4$ and $\pgl_2(p^k)$ is $O(\xi(\log \log q +1) +\mu( k\log \log q \log q + \log q))$.

Observe that the algorithm presented in Section \ref{sec:algo} together with Remarks \ref{rem:alt4} and \ref{rem:the:killer} and the computation of the complexity above gives a proof of Theorem \ref{pgl2} and Corollaries \ref{psl2} and \ref{corall}.

\end{document}